\documentclass[11pt]{amsart}
\usepackage{geometry}                
\usepackage[parfill]{parskip}    
\usepackage{graphicx}
\usepackage{subfig}
\usepackage{psfrag}
\usepackage{labelfig}
\usepackage{amssymb}
\usepackage{epstopdf}
\usepackage{color}
\newtheorem{theorem}{Theorem}[section]    
\newtheorem{lemma}[theorem]{Lemma}          
\newtheorem{proposition}[theorem]{Proposition}  
  
\theoremstyle{definition}
\newtheorem{definition}[theorem]{Definition}

\newtheorem{example}[theorem]{Example}    
\newtheorem*{remark}{Remark}             
\numberwithin{equation}{section}

\newcommand{\e}{\varepsilon}

\newcommand{\F}{\mathcal F_{ob} }

\newcommand{\Z}{\mathbb{Z}}

\newcommand{\MCG}{\mathcal{MCG}}

\title{On the self-linking number of transverse links}

\author{Tetsuya Ito}
\address{Research Institute for Mathematical Sciences, Kyoto university
Kyoto, 606-8502, Japan}
\email{tetitoh@kurims.kyoto-u.ac.jp}
\urladdr{http://www.kurims.kyoto-u.ac.jp/~tetitoh/}

\author{Keiko Kawamuro}
\address{Department of Mathematics \\ 
The University of Iowa \\ Iowa City, IA 52242, USA}
\email{kawamuro@iowa.uiowa.edu}
\date{\today}

\begin{document}

\begin{abstract}
We review the braid theoretic self-linking number formula in \cite{ik1-1} and study its applications. 
\end{abstract} 

\maketitle
\section{Review of the self-linking number}

Let $M$ be an oriented closed $3$-manifold equipped with a contact structure $\xi$. 
Let $K \subset (M, \xi)$ be an oriented transverse link, namely at every point $p \in K$ the link $K$ is transverse to the contact plane $\xi_p$ positively. 

The self-linking number is an invariant of null-homologous transverse links. 
Suppose that a transverse link $K \subset (M, \xi)$ bounds a Seifert surface $\Sigma \subset M$. 
Choose a nowhere vanishing section $s$ of the rank $2$ vector bundle $\xi|_\Sigma \to \Sigma$ and push $K$ into the direction of $s$ to obtain a copy of $K$, denoted by $K^s$. 
The self-linking number $sl(K, [\Sigma])$ is defined to be the algebraic intersection number of $K^s$ and $\Sigma$. 
In other words $sl(K, [\Sigma]) = - \langle e(\xi), [\Sigma]\rangle$, the evaluation of the Euler class $e(\xi)$ over $[\Sigma]\in H_2(M, K; \Z)$.

In \cite{ik1-1} we study a braid theoretic formula of the self-linking number. 
In this paper we will give applications and observations of the formula. 
In order to review the formula, we recall the following two fundamental results. 

The first one is the Giroux correspondence \cite{giroux3}: 
Given an oriented closed three manifold $M$, there is a one to one correspondence between the set of contact structures on $M$ up to contact isotopy and the set of open book decompositions of $M$ up to positive stabilizations.

The second result is due to Bennequin \cite{Ben}, Mitsumatsu and Mori \cite{MM}, and Pavelescu \cite{P, P2}:  
Suppose that the contact manifold $(M, \xi)$ is corresponding to (in the sense of Giroux) the open book $(S, \phi)$, where $S$ is a compact oriented surface with non-empty boundary and $\phi:S\to S$ is a diffeomorphism fixing the boundary of $S$ point-wise.  
(In the following, we say that $(M, \xi)$ is {\em supported by} $(S, \phi)$.) 
There is a one to one correspondence between the set of transverse links in $(M, \xi)$ up to transverse isotopy and the set of braids with respect to $(S, \phi)$ up to braid isotopy and positive stabilization.

Here by a {\em braid with respect to} $(S, \phi)$ we mean an oriented link in $M=M_{(S, \phi)}$ that intersects every page of $(S, \phi)$ positively and is never tangent to any pages.   
A close connection between transverse links and braids was first observed by Bennequin \cite{Ben} for the case $(S, \phi) = (D^{2},id)$. 
The general case was studied by Mitsumatsu and Mori \cite{MM} and Pavelescu \cite{P, P2}.

In order to state the self-linking number formula we fix notations here. 
Let $S=S_{g, r}$ be a surface of genus $g$ and with $r$ boundary components, $(S, \phi)$ an open book, and $\widehat b$ a null-homologous $n$-stranded braid with respect to  $(S, \phi)$. 
We may identify the group of $n$-stranded braids with respect to  $(S, \phi)$ with the fundamental group of the configuration space of $S$ with $n$ punctures.  
We place $n$ points near one of the boundary components of $S$ and choose the generators $\{ \sigma_1, \ldots, \sigma_{n-1}, \rho_1, \ldots, \rho_{2g+r-1} \}$ of the braid group as depicted in Figure~\ref{braid-generator}. 
\begin{figure}[htbp] 
\begin{center}
\SetLabels
(.81*.38) $\rho_1$\\
(.54*.4) $\rho_{r-1}$\\
(.49*.6) $\rho_{r}$\\
(.4*.8)  $\rho_{r+1}$\\
(.3*.4) $\rho_{2g+r-2}$\\
(.17*.85) $\rho_{2g+r-1}$\\
(.32*.1) $\sigma_i$\\
(.16*.35) $1$\\
(.29*.05) $i$\\
(.38*.04) $i+1$\\
(.51*.16) $n$\\
\endSetLabels
\strut\AffixLabels{\includegraphics*[height=4.5cm]{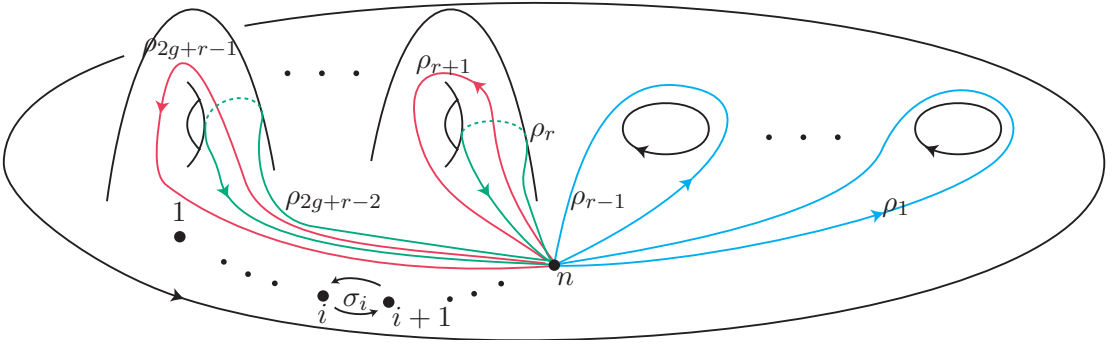}}
\caption{Generators of the braid group.}\label{braid-generator}
\end{center}
\end{figure}

Suppose that the closed braid $\widehat b$ is represented by a braid word 
$$b_1^{\e_1} b_2^{\e_2} \cdots b_l^{\e_l}$$ 
where $b_i \in \{ \sigma_1, \ldots, \sigma_{n-1}, \rho_1, \ldots, \rho_{2g+r-1} \}$ and $\e_i \in \Z$. 
Let $[\rho_i] \in H_1(S; \Z)$ denote the homology class of the loop $\rho_i$ and put $[\sigma_j]:= 0 \in H_1(S; \Z)$. 
Define 
$$[b] := \e_1 [b_1] + \cdots + \e_l [b_l] \in H_1(S;\Z).$$ 
We note that $[b]$ is an element of $H_1(S; \Z)$ and different from $[\widehat b] = 0 \in H_1(M; \Z)$.

The null-homologous assumption on $\widehat b$ implies the following.

\begin{lemma}\label{lemma-a}\cite[Claim 3.8]{ik1-1}
There exists (not necessarily unique) a homology class $a\in H_1(S, \partial S;\Z)$ such that 
\begin{equation}
\label{eqn:Seifert}
[b] = a - \phi_*(a)  \in H_{1}(S;\Z).
\end{equation}  
\end{lemma}
It is implicit in the proof of \cite[Claim 3.8]{ik1-1} that the homology class $a \in H_1(S, \partial S;\Z)$ is represented by properly embedded circles and arcs that  do not intersect the $n$ punctures of $S$. 
Since $\phi=id$ near $\partial S$ we may regard $a - \phi_*(a) \in H_{1}(S;\Z)$.

The homology class $a\in H_1(S, \partial S;\Z)$ depends on choice of a  Seifert surface $\Sigma$ of $\widehat{b}$ and is represented by $S_1 \cap \Sigma$,   the intersection of $\Sigma$ with the $t=1$ page, $S_1$, of the open book. 
Thus $a$ is uniquely determined by the homology class $[\Sigma] \in H_{2}(M,\widehat{b})$. 
Conversely, in Theorem \ref{theorem:sl-formula} below we construct a Seifert surface $\Sigma_{a}$ from $a$, i.e., the homology class $[\Sigma_{a}]\in H_{2}(M,\widehat{b})$ is uniquely determined by $a$.

Let 
$a^{\rm{NF}}$ 
denote the {\em normal form} of $a\in H_1(S, \partial S;\Z)$, that is a properly embedded multi-curve representing $a$. 
See \cite[Definition 3.3]{ik1-1}) for the precise definition of the normal form, where we use the notation $N(a)$ instead of $a^{\rm NF}$.
Let $F$ be a cobordism surface from $\phi(a^{\rm NF})$ to $(\phi_* (a))^{\rm NF}$, i.e., a compact, oriented, properly embedded surface in $S \times [0, 1]$ such that 
\begin{itemize}
\item 
$F \cap (S \times \{0\})  = \partial F \cap (S \times \{0\}) = - \phi(a^{\rm NF}) \times \{0\}$.
\item 
$F \cap (S \times \{1\}) =  \partial F \cap (S \times \{1\}) =(\phi_*(a))^{\rm NF} \times \{1\}$.
\item 
$\partial (\phi(a^{\rm NF}))= \phi(a^{\rm NF}) \cap \partial S = (\phi_*(a))^{\rm NF} \cap \partial S= \partial ((\phi_*(a))^{\rm NF})$. 
\item 
$\partial F = (-\phi(a^{\rm NF}) \times \{0\}) \cup ((\phi_*(a))^{\rm NF} \times \{1\}) \cup (\partial (\phi(a^{\rm NF})) \times [0,1])$.
\item All the singularities of the singular foliation $\{(S \times \{t\}) \cap F \ | \ t \in [0,1]\}$ on $F$ are of hyperbolic type.
\end{itemize}
Such a surface $F$ exists (not necessarily uniquely) and is called an {\em OB cobordism} in \cite[Definition 3.4]{ik1-1}.

Let $h_\pm (F)$ denote the number of $\pm$ hyperbolic singularities of the above mentioned singular foliation on $F$. 
It is proven in \cite[Prop 3.6]{ik1-1} that the algebraic count of the hyperbolic points $h_+(F) - h_-(F)$ is uniquely determined by $[\phi]$ and $a$. 
Hence the following function $c([\phi], a)$ is well-defined. 

\begin{definition}\label{def of function c}
Let $[\phi] \in \MCG(S)$ and $a \in H_{1}(S,\partial S)$. 
Define
\[ c([\phi], a) := h_+(F) - h_-(F). \]
\end{definition} 

\begin{theorem}[Self linking number formula \cite{ik1-1}]\label{theorem:sl-formula}

Let $\widehat b$ be an $n$-stranded closed braid with respect to the open book $(S, \phi)$. Suppose that $b=b_1^{\e_1} b_2^{\e_2} \cdots b_l^{\e_l}$ $(\e_i \in \Z)$ is a braid word of $\widehat b$.  
Let $[b] \in H_1(S;\Z)$ and $a \in H_1(S, \partial S;\Z)$ be as above. 
There exists a Seifert surface $\Sigma=\Sigma_a$ of $\widehat{b}$ such that the self-linking number satisfies the formula 
\begin{equation}\label{eq:self-linking}
sl(\widehat{b}, [\Sigma]) = -n + \widehat{\exp}(b) - \phi_*(a)\cdot[b] + c([\phi] , a).
\end{equation}
Here $\phi_*(a) \cdot [b]$ represents the intersection pairing 
$H_{1}(S,\partial S) \times H_{1}(S) \rightarrow \Z$ and $\widehat{\exp}(b)$ is a generalized exponent sum defined by
$$\widehat{\exp}(b)= \sum_{i=1}^l \e_i-\sum_{1\leq j<i \leq l} \e_i \e_j [b_j] \cdot [b_i].$$
\end{theorem}

As discussed in \cite[Propositions 3.20 and 3.21]{ik1-1} it turns out that $c([\phi],a)$ has a decomposition $c([\phi], a) = c'([\phi], a) - 2k([\phi], a)$.

The first term $c'([\phi], a)$ carries homological information: 
With a basis $\{[\rho'_{1}],\ldots, [\rho'_{2g+r-1}]\}$ of $H_{1}(S,\partial S)$ that is  `dual' to $\{[\rho_{1}],\ldots, [\rho_{2g+r-1}]\}$ of Figure~\ref{braid-generator}, $c'([\phi], a)$ is determined by the homology intersections of $[\rho'_{i}]$ and $[\phi_{*}(\rho'_{i})]$. 

The second term $k(\phi,a)$ is more interesting.
It is a crossed homomorphism, i.e.,
\begin{gather*}
\begin{cases}
k([\psi\phi], a)= k([\phi], a) + k([\psi], \phi_* (a)), \\
k([\phi], a+a') = k([\phi], a) + k([\phi], a').
\end{cases}
\end{gather*}
Moreover, $k(\phi,a)$ represents a non-trivial cohomology class in $H^{1}(\MCG(S),H^{1}(S))$ that appears in various contexts in mapping class group theory. 
In our setting $k([\phi], a)$ coincides with Morita's version of the crossed homomorphism in \cite[\S 6]{mo}.

Suppose that $S$ has connected boundary. 
Let $H:= H^1(S;\Z) \simeq H_1(S,\partial S; \Z)$. 
In \cite{mo3} Morita proves the existence of a crossed homomorphism $\kappa: \MCG(S) \rightarrow \frac{1}{2} \wedge^{3} H$ and shows that $\kappa$ is the unique extension of the Johnson homomorphism $\tau: \mathcal I \to \wedge^3 H$ up to coboundaries for $H^1(\MCG(S), \wedge^3 H)$, where $\mathcal I$ is the Torelli group for $S$. 
Our map $k$ is obtained by $k=C\circ \kappa$, that is  composition of $\kappa$ and the contraction $C: \wedge^3 H \to H;$ $x\wedge y \wedge z \mapsto 2((x\cdot y)z+(y\cdot z)x + (z\cdot x) y)$.

\subsection{The disk open book}\label{subsec:disk}

If $(S, \phi)=(D^2, id)$ then the compatible contact structure is the standard contact $3$-sphere and the braid group is generated by the transpositions $\{ \sigma_1, \ldots, \sigma_{n-1}\}$. 
In \cite{Ben} Bennequin proves that for an $n$-braid $b=\sigma_{i_1}^{\e_1} \sigma_{i_2}^{\e_2} \cdots \sigma_{i_l}^{\e_l}$,  ($\e_i = 1$ or $-1$) 
\begin{equation}\label{sl-number-for-disk}
sl(\widehat{b}, [\Sigma]) = -n + \exp(b)
\end{equation}
where $\exp(b)= \sum_{i=1}^l \e_i$ is the usual exponent sum of the braid word (note $[\sigma_i]=0$ so $\exp(b)=\widehat{\exp}(b)$) and $\Sigma$ is the so called {\em Bennequin surface} of $b$.  
The Bennequin surface consists of $n$ parallel stacked disks positively pierced by the binding (the braid axis) and joined by twisted bands. 
More precisely, each positive (resp. negative) braid word $\sigma_i$ (resp. $\sigma^{-1}$) corresponds to a positively (resp. negatively) twisted band joining the  $i^{\rm th}$ and the $i+1^{\rm st}$ disks.

\begin{remark}
Since $H_2(S^3; \Z) = 0$ the self-linking number does not depend on choice of a Seifert surface of $\widehat b$ and it makes sense to denote it simply by $sl(\widehat b)$. 
We denote $sl(\widehat b, [\Sigma])$ in (\ref{sl-number-for-disk}), instead of $sl(\widehat b)$, in order to remember that our particular construction of $\Sigma=\Sigma_a$ in Theorem~\ref{theorem:sl-formula} is a generalization of the Bennequin surface. 
\end{remark}

\subsection{Annulus open books} 
Let $(S, \phi)=(A, T^k)$ be the annulus open book with the $k^{\rm th}$ power of the positive Dehn twist $T$ about a core circle of an annulus $A$. 
In \cite{KP} it is proven that 
$$sl(b, [\Sigma])= -n + \exp(b) - \phi_*(a)\cdot [b]$$
for some Seifert surface $\Sigma$ that is a generalization of the Bennequin  surface. 
\begin{remark}
When $k\neq 0$ we remark that the self-linking number does not depend on choice of a Seifert surface because the second homology of the manifold vanishes. 
\end{remark}

\subsection{Planar open books}
When $S$ is a planar surface it is shown in \cite{K} that
$$sl(b, [\Sigma])= -n + \exp(b) - \phi_*(a)\cdot [b] + c'([\phi], a).$$

\section{Applications}

\subsection{Positive braids and Bennequin-Eliashberg inequality} 

One of the celebrated results in contact topology is the Bennequin-Eliashberg inequality.

\begin{theorem}[The Bennequin-Eliashberg inequality \cite{el2}]
\label{theorem:BEinequality}
If a contact 3-manifold $(M,\xi)$ is tight then for any null-homologous transverse link $L$ and its Seifert surface $\Sigma$ we have 
\[ sl(L,[\Sigma]) \leq -\chi(\Sigma). \]
\end{theorem}

It is interesting to ask when the Bennequin-Eliashberg (BE) inequality is sharp. 

We say that a closed braid $\widehat b$ with respect to the open book $(D^2, id)$ is {\em strongly quasi-positive} if $\widehat b$ is represented by a braid word 
$\sigma_{i_1, j_1} \cdots \sigma_{i_l, j_l}$ where 
$$\sigma_{i, j}= (\sigma_i \cdots \sigma_{j-2}) \sigma_{j-1} (\sigma_i \cdots \sigma_{j-2})^{-1}$$ and $i<j$. 
Clearly a positive braid is strongly quasi-positive. 
A strongly quasi-positive braid bounds a particular Seifert surface which consists of $n$ disks pierced by the binding and the $i^{\rm th}$ and the $j^{th}$ disks are joined by a positively twisted band for each $\sigma_{i, j}$ in the braid word $\sigma_{i_1, j_1} \cdots \sigma_{i_l, j_l}$. 
It is easy to verify that the BE inequality is sharp on a strongly quasi-positive braid with the above mentioned Seifert surface.

However, in general, this is not the case if $(S, \phi)\neq (D^2, id)$:

Consider the simplest non-trivial open book, $(A, T^k)$,  the annulus open book with the $k^{\rm th}$ $(k\geq 0)$ power of the positive Dehn twist $T$. The contact structure $\xi_{(A, T^k)}$ is tight. 
If a closed braid $\widehat b$ with respect to  $(A, T^k)$ is null-homologous then by Lemma~\ref{lemma-a} we have $[b]= sk[\rho] \in H_1(S, \Z)$ 
for some $s\in\Z$. 
Suppose that $\widehat b$ is a {\em positive} braid, that is, $\widehat b$ is represented by a positive braid word $b=b_1\cdots b_l$ in $\{\sigma_1, \ldots, \sigma_{n-1}, \rho\}$. 
Positivity of $\widehat b$ implies $s \geq 0$.  
Let $\Sigma$ be a Seifert surface of $\widehat b$ constructed in the manner  discussed in \S 3.2 of \cite{ik1-1}. 

\begin{proposition}\label{prop:sharpness}
The Bennequin-Eliashberg inequality is sharp on $(\widehat b, \Sigma)$ if and only if $s=0$ or $s=k=1$. 
\end{proposition}

In the case of $s=0$ the positive word $b_1\cdots b_l$ does not contain $\rho$. Hence the braid lies near a small tubular neighborhood of one of the binding components and $\Sigma$ is nothing but the Bennequin surface explained in \S~\ref{subsec:disk}. 

In the case of $s=k=1$ the open book is a positive stabilization of $(D^2, id)$ and the braid $b$ is a positive stabilization of some positive braid in $(D^2, id)$. 

Therefore Proposition~\ref{prop:sharpness} implies that, except for the two trivial cases mentioned above, the BE-inequality is not sharp on positive braids. 

In the following proof we use techniques of open book foliations studied in \cite{ik1-1}. 

\begin{proof}
In terms of $e_{\pm}, h_{\pm} \in\Z$, the numbers of $\pm$ elliptic/hyperbolic singularities of the open book foliation $\F(\Sigma)$, the BE-inequality is restated as
$$-(e_+ - e_-) + (h_+ - h_-) \leq -(e_+ + e_-) + (h_+ + h_-),$$ 
that is,  $e_- \leq h_-.$

By our specific construction of $\Sigma$ introduced in the proof of \cite[Theorem 3.10]{ik1-1} the normal form $a^{NF}$ of the homology class $a$ is parallel essential $s$ arcs in the annulus $A$ connecting the two  boundary components. 
The endpoints of each arc become elliptic points with distinct signs of the open book foliation $\F(\Sigma)$. 
Therefore, we have $s$ negative elliptic points, that is, 
$$e_- = s.$$
Also the proof of \cite[Theorem 3.10]{ik1-1} shows that  the term $\phi_*(a) \cdot [b]$ in the self-linking number formula (\ref{eq:self-linking}) is exactly the number of negative hyperbolic points of the open book foliation $\F(\Sigma)$ provided $\widehat b$ is a positive braid. 
We have 
$$h_- = \phi_*(a) \cdot [b] = (s[\rho']) \cdot (sk[\rho])=s^2 k,$$
where $\rho'$ is an essential arc connecting the two  boundary components of $A.$ 

The inequality $s \leq s^2 k$ holds if and only if $s=0$ or $s=k=1$. 
\end{proof}

The next is an example of a non-positive braid on which the BE-inequality is sharp. 

\begin{example}
\label{exam:BEsharp}
Let $(S_{g, 1}, \phi)$ be a general open book with connected binding. 
We do {\em not} assume the contact structure $\xi_{(S, \phi)}$ is tight. 
Fix $n \geq 1$. 
Let $\widehat b$ be an $n$-braid represented by the braid word:
$$b= \sigma_1 \cdots \sigma_{n-1} [\rho_1, \rho_2] \cdots [\rho_{2g-1}, \rho_{2g}]$$ 
where $[a, b] = aba^{-1}b^{-1}$, so $b$ is not a positive word.  
The braid $\widehat b$ is a positive push-off of the binding and the algebraic intersection  number of $\widehat b$ and a page surface $S$ is $n$. 
Since $[b]=0 \in H_1(S, \Z)$ we can choose $a=0 \in H_1(S, \partial S; \Z)$. 
Then (\ref{eq:self-linking}) implies that $$sl(\widehat{b}, [\Sigma]) = -n + \widehat{\exp}(b) = -n + (n-1+2g) = - \chi(S),$$ 
i.e., the BE-inequality is sharp on $\widehat b$. 
\end{example}

More general results by Hedden \cite{H} (for the standard tight $3$-sphere) and Etnyre and Van~Horn-Morris \cite{EV} (for general tight contact manifolds) have been known:  
Let $(M, \xi)$ be a {\em tight} contact structure. 
Assume that a manifold $M$ admits an open book decomposition $(S, \phi)$. 
Let $K$ be the binding and $\Sigma$ a page of the open book. 
Then $sl(K,[\Sigma]) = -\chi(\Sigma)$ if and only if $\xi$ is supported by $(S, \phi)$ or obtained from $\xi_{(S, \phi)}$ by adding Giroux torsion.

\subsection{Permutation of a braid word}

Let $\widehat b$ and $\widehat b'$ be closed braids with respect to  $(S, \phi)$ represented by the braid words 
\begin{gather*}
\begin{cases}
b = b_1 \cdots b_k, \\ 
b'= b_{\tau(1)} \cdots b_{\tau(k)},
\end{cases}
\end{gather*}
where $b_i \in \{ \sigma_1^\pm, \ldots, \sigma_{n-1}^\pm, \rho_1^\pm, \ldots, \rho_{r-1}^\pm\}$ and $\tau\in \mathfrak S_k$ is a permutation of $k$ letters.

Assume that $\widehat b$ is null-homologous. 
Let $a \in H_{1}(S,\partial S)$ be a homology class satisfying (\ref{eqn:Seifert}) i.e., $[b]=[b']=a-\phi_*(a)$. 
Thus we have Seifert surfaces $\Sigma$ and $\Sigma'$ of $\widehat{b}$ and $\widehat{b'}$, respectively, defined by the same homology class $a$.

\begin{proposition}
Assume that $(S,\phi)$ is a planar open book. 
Then the self-linking numbers of the two closed braids $\widehat b$ and $\widehat b'$ are equal:
$$ sl(\widehat b, [\Sigma]) = sl(\widehat b', [\Sigma']).$$
\end{proposition}

\begin{proof}
We have $$\widehat\exp(b)= \exp(b)= \exp(b')= \widehat\exp(b')$$ where the first and the third equalities follow from the planar condition of $S$ and the second equality holds because the braid words $b$ and $b'$ are related to each other by a permutation. 
Since $[b]=[b'] \in H_1(S; \Z)$ 
we may take the same $a \in H_1(S, \partial S)$ for $b$ and $b'$. By Theorem~\ref{theorem:sl-formula} we get $sl(\widehat b, [\Sigma]) = sl(\widehat b', [\Sigma'])$ for some Seifert surfaces $\Sigma$ and $\Sigma'$.
\end{proof}

The assumption that $S$ is planar is necessary: 
Let $(S_{g,1},id)$ be a non-planar open book ($g>0$). 
Consider the braid $b=[\rho_{1},\rho_{2}]= \rho_{1}\rho_{2}\rho_{1}^{-1}\rho_{2}^{-1}$ and its permutation
$b'=\rho_1 \rho_1^{-1}\rho_2 \rho_{2}^{-1}$.
Since $[b]=[b']=0$ we may take $a=0$.
Clearly $\widehat{b'}$ is the unknot and $sl(\widehat{b'},[\Sigma'])=-1$, whereas $sl(\widehat{b},[\Sigma])= +1$ as discussed in Example~\ref{exam:BEsharp} .

\subsection{Contact surgery preserving the self-linking number}

For a Legendrian knot $L$ in a contact manifold $(M,\xi)$ one performs $(\pm 1)$-surgery along $L$ with respect to the contact framing \cite{dg}. Namely, the contact structure $\xi$ on 
the complement of $L$ 
extends to a contact structure $\xi_{L^{\pm 1}}$ on the surgered manifold $M_{L^{\pm 1}}$.

\begin{proposition}
Let $L$ be a Legandrian knot and $K$ a null-homologous transverse knot with a Seifert surface $\Sigma$ in $(M,\xi)$. 
If the algebraic intersection number of $\Sigma$ and $L$ is zero, then in $(M_{L^{\pm 1}},\xi_{L^{\pm 1}})$ there exists a Seifert surface $\Sigma_{L^{\pm 1}}$ for the corresponding transverse knot $K_{L^{\pm 1}}$  such that
\[ sl(K, [\Sigma]) = sl(K_{L^{\pm 1}}, [\Sigma_{L^{\pm 1}}]).\]
\end{proposition}

\begin{proof}
Take an open book decomposition $(S,\phi)$ compatible with $(M,\xi)$ such that the Legendrian knot $L$ sits on the  page $S_1(=S_{t=1})$ and the contact framing of $L$ and the surface framing of $L$ induced by $S_1$ agree \cite{AO, Pl}. 
The open book $(S,\phi \circ T_{L}^{\mp 1})$ supports $(M_{L^{\pm 1}},\xi_{L^{\pm 1}})$, where $T_{L}$ denotes the right-handed Dehn twist along $L$.  
In particular, if a closed braid $\widehat{b}$ with respect to  $(S, \phi)$ represents the transverse knot $K$ then the same braid $\widehat{b}$ with respect to  $(S,\phi \circ T_{L}^{\mp 1})$ represents $K_{L^{\pm 1}}$.

Let $a := [\Sigma \cap S_1] \in H_{1}(S,\partial S)$. 
Since $\phi_*(a)=[\Sigma\cap S_0]$ and $\partial \Sigma =  \widehat{b}$ we have $[b]=a - \phi_*(a)$. 
Since $L \subset S_1$ we have $a \cdot [L]= [\Sigma] \cdot [L] = 0$. 
Thus, 
${T_{L}^{\mp 1}}_*(a) = a$ and 
\begin{equation}\label{eq-1}
\phi_*(a)= \phi_* ({T_{L}^{\mp 1}}_*(a))=(\phi \circ T_{L}^{\mp 1})_*(a). 
\end{equation}
We have 
\begin{equation}\label{eq-2}
c([\phi \circ T_{L}^{\mp 1}], a) = 
c([T_L^{\mp 1}], a) + c([\phi], {T_L^{\mp 1}}_* (a)) =c([\phi], {T_L^{\mp 1}}_* (a)) = c([\phi], a)
\end{equation}
where the first (resp. the second) equality holds by Item (2) (resp. (4)) of \cite[Proposition~3.12]{ik1-1}.

Finally, Theorem~\ref{theorem:sl-formula} together with (\ref{eq-1}) and (\ref{eq-2}) implies that there exists a Seifert surface, $\Sigma_{L^{\pm 1}}$, such that $sl(K, [\Sigma]) = sl(K_{L^{\pm 1}}, [\Sigma_{L^{\pm 1}}])$.
\end{proof}

\subsection{Twisting by elements in the Johnson kernel}

Let $S=S_{g, 1}$ be a surface with genus $g \geq 3$ and connected boundary,  and $\widehat b$ a closed braid with respect to the open book $(S, \phi)$. 
For an element of the Torelli group $\psi \in \mathcal{I}_{g}$, let $\widehat b_{\psi}$ be the closed braid with respect to  the open book $(S, \psi\phi)$ such that $\widehat b$ and $\widehat b_{\psi}$ are represented by the same braid word, $b$, in $\{ \sigma_1^\pm, \ldots, \sigma_{n-1}^\pm, \rho_1^\pm, \ldots, \rho_{2g}^\pm \}$. 

Assume that $\widehat{b}$ is null-homologous. 
By Lemma~\ref{lemma-a} there exists $a \in H_{1}(S,\partial S)$ with $\phi_*(a)-a = [b]$. 
Since the Torelli group acts on $H_{1}(S,\partial S)$ trivially, $a$ satisfies 
$$(\psi \cdot \phi)_*(a) -a = \phi_*(a)-a = [b].$$
In particular, the proof of \cite[Claim 3.8]{ik1-1} implies that $\widehat b_{\psi}$ is also null-homologous.
Recall the {\em Johnson kernel} $\mathcal{K}_{g}$, the kernel of the Johnson homomophism  $\tau: \mathcal I_g \to \wedge^3 H$ where $H=H_1(S;\Z)$.

\begin{proposition}
\label{prop:JK}
Let $\Sigma$ and $\Sigma_{\psi}$ be Seifert surfaces constructed by using the above $a \in H_{1}(S,\partial S)$ and the method as in the proof of \cite[Theorem 3.10]{ik1-1}.  
If $\psi \in \mathcal{K}_{g}$, then 
$$sl(\widehat b, [\Sigma]) = sl(\widehat b_{\psi}, [\Sigma_{\psi}]).$$
\end{proposition}

\begin{proof}

This essentially follows from the formula $c(\phi,a) = c'(\phi,a)- 2k(\phi,a)$ and the fact that $k$ is a contraction of the Johnson homomorphism, but here we give a more direct proof that does not use this explicit formula. 

The Johnson kernel $\mathcal K_g$ is generated by Dehn twists about separating simple closed curves when $g \geq 3$, see \cite{J}. 
We may assume that $\psi = T_{C_l} \circ \dots \circ T_{C_1}$ for some separating simple closed curves $C_1, \dots, C_l$. 
Since $S$ has connected boundary the separating property implies that $x \cdot [C_i]=0$ for any $x \in H_1(S, \partial S)$. 
Moreover, by \cite[Proposition 3.12-(4)]{ik1-1} we have $c([T_{C_i}], x)=0$.
Repeatedly using the crossed homomorphism property of the function $c$, proven in \cite[Proposition 3.12-(2)]{ik1-1}, we have 
$$
c([\psi], \phi_* (a))=c([T_{C_1}], a) + c([T_{C_l} \circ \dots \circ T_{C_2}], {T_{C_1}}_*(a)) = \dots = 0,
$$
and $$c([\psi\phi], a) = c([\phi], a) + c([\psi], \phi_*(a))= c([\phi], a).$$ 
Therefore, by Theorem~\ref{theorem:sl-formula} we have $sl(\widehat b, [\Sigma]) = sl(\widehat b_{\psi}, [\Sigma_{\psi}])$. 
\end{proof}

As the definition of $c([\phi],a)$ shows, Proposition~\ref{prop:JK} does not hold if we simply assume $\psi \in \mathcal{I}_{g}$.

\subsection{Twisting along a binding component}

Let $\widehat b$ be a null-homologous $n$-braid in the open book $(S, \phi)$ where $S \neq D^2$. 
Let $\Sigma$ be a Seifert surface of $\widehat b$ admitting an open book foliation $\F(\Sigma)$. 
By braid isotopy, we may assume that 
\begin{itemize}
\item
the $n$ points $\widehat b \cap S_1$ are all very close to a binding component, $C$, and 
\item 
each of the $n$ points is joined with $C$ by an a-arc in the foliation $\F(\Sigma)$. 
\end{itemize}
Fix a number $k\in\Z$. 
Let $\tau_C$ be the braid word that corresponds to the $n$ points winding around $C$ together (like the usual {\em full twist} of $n$-braids in braid theroy). 
Let $\widehat b'$ be the braid closure of the braid word, 
$$b'=b \cdot (\tau_C)^{k} \quad (\mbox{read from the right to left}),$$
living in the open book $(S, {T_C}^k \circ \phi)$ where $T_C$ is the right-handed Dehn twist along $C$.  
We may consider that $\widehat b'$ is a natural extension of the original braid $\widehat b$ under adding Dehn twist ${T_C}^k$ to the monodromy $\phi$ of the original open book. 

Let $\bf a_+$ be the number of positive elliptic points in the open book foliation $\F(\Sigma)$ that are sitting on $C$ as the end points of a-arcs. 
Let $\bf b_\pm$ be the number of $\pm$ elliptic points of $\F(\Sigma) \cap C$ that are the end points of b-arcs. These numbers satisfy 
$$\Sigma \cdot C ={\bf a}_+ + {\bf b}_+ - {\bf b}_-.$$ 

\begin{proposition}
Let $\widehat b$ (resp. $\widehat b'$) be the above closed braid with respect to the open book $(S, \phi)$ (resp. $(S, {T_C}^k \circ \phi)$). 
Assume that 
$${\bf a}_+=  \Sigma \cdot C = n.$$ 
Then there exists a Seifert surface $\Sigma'$ of $\widehat b'$ that is a natural extension of $\Sigma$ and satisfies 
$$sl(\widehat b',  [\Sigma']) =  sl(\widehat b, [\Sigma]).$$
\end{proposition}

\begin{proof}
Fix a small $\e>0$. 
We may assume that for $0\leq t \leq \e$ the multi-curve $\Sigma \cap S_t$ is constant. 
We will define a Seifert surface $\Sigma'$ of $\widehat b'$ by modifying $\Sigma$. 

For $\e \leq t \leq 1$ we define $\Sigma'$ by 
$$\Sigma' \cap S_t = \Sigma \cap S_t .$$ 

When $t=0$ we define $\Sigma'$ by 
$$\Sigma' \cap S_0 = {T_C}^k (\Sigma \cap S_0) = ({T_C}^k  \circ\phi) (\Sigma' \cap S_1).$$
That is, in a small collar neighborhood $C \times I$ of $C$ in the page $S_0$, all the arcs ending at $C$ are winding around $C$ $k$ times. 
See the movie presentation depicted in Figure~\ref{CxI}. 
\begin{figure}[htbp] 
\begin{center}
\SetLabels
(0*1) $C\times\{1\}$\\
(.4*1) $C\times\{0\}$\\
(.2*.75) a-arc\\
(.22*.7) $t=0$\\
(.78*.7) $t=\e/3$\\
(.78*.33) $t=2\e/3$\\
(.22*-.02) $t=\e$\\
(.45*.75) b-arc\\
(.45*.95) b-arc\\
\endSetLabels
\strut\AffixLabels{\includegraphics*[height=15cm]{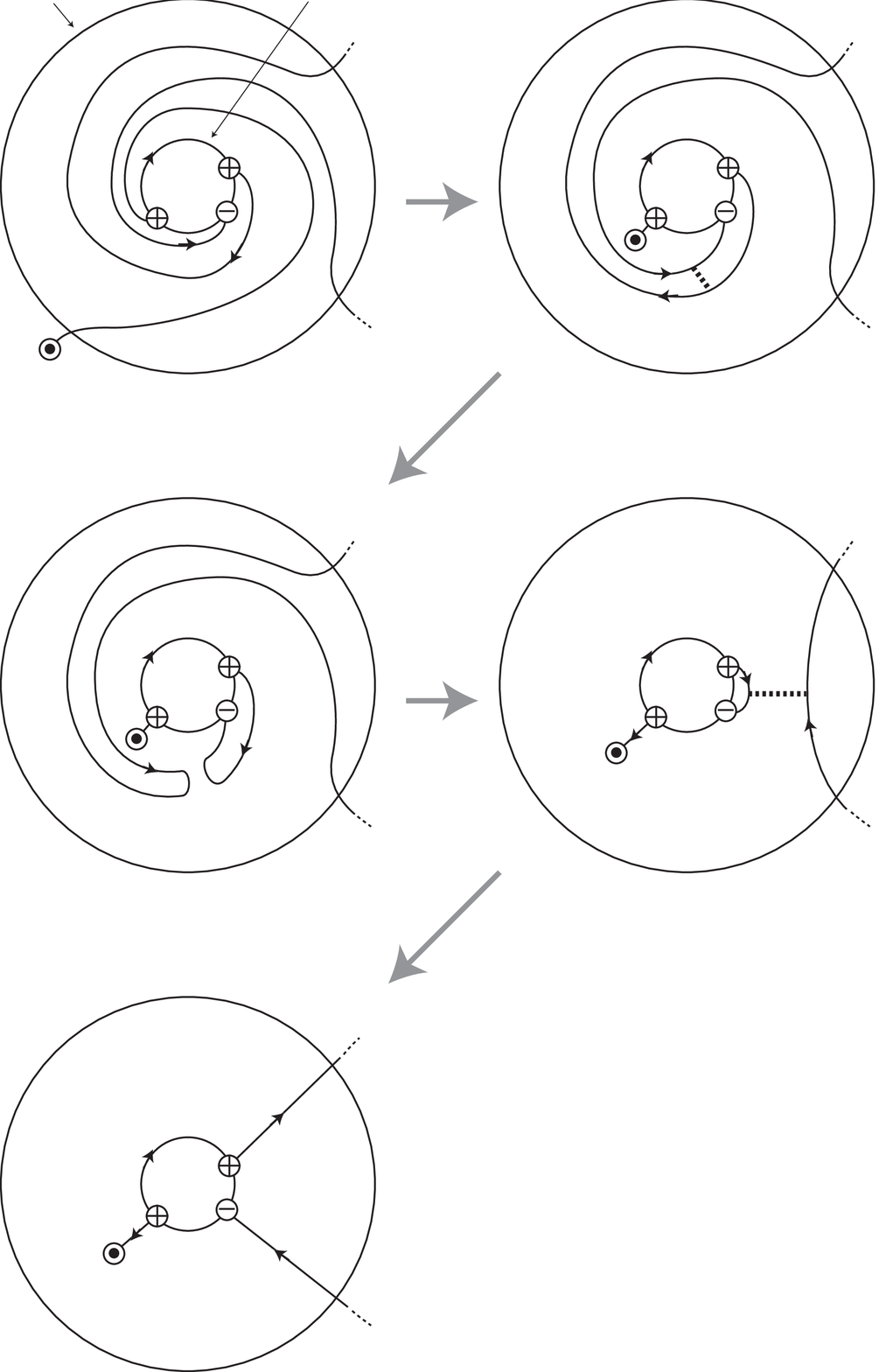}}
\caption{Construction of $\Sigma'$ for $0\leq t\leq \e$ in the region $C\times I$, where $k=1$ and $m=1$.
The hyperbolic point at $t=\e/3$ is negative and at $t=2\e/3$ is positive.}\label{CxI}
\end{center}
\end{figure}

For $0<t <\e/3$ by isotopy we unwind all the a-arcs in the multi-curve $\Sigma' \cap S_0$ ending at $C$. 
See the first row of Figure~\ref{CxI}. 
In the page $S_{\e/3}$ all the a-arcs in $C\times I$ are very close to $C\times \{0\}$. 
The rest of the a-arcs and b-arcs (including the ones ending at $C$) stay fixed. 
Since ${\bf a}_+ = n$ this construction corresponds to the added term $(\tau_C)^k$ in the word $b'=b \cdot (\tau_C)^k$.

The condition ${\bf a}_+=  \Sigma \cdot C$ implies ${\bf b}_+ = {\bf b}_-$. 
Put $m= {\bf b}_+ = {\bf b}_-$. 
In the page $S_{\e/3}$ the region $C\times I$ contains $2m$ b-arcs connecting  points $P_i = \{p_i\}\times \{0\}$ and $Q_i= \{p_i\}\times \{1\}$, where $i=1, \dots, 2m$ and $P_i$ is an elliptic point, and winding $k$ times. See Figure~\ref{CxI-2}. 
\begin{figure}[htbp] 
\begin{center}
\SetLabels
(-.01*1) $C\times\{0\}$\\
(-.01*.72) $C\times\{1\}$\\
(.22*.68) $t=\e/3$\\
(.22*.33) $t=2\e/3$\\
(.22*-.03) $t=\e$\\
(.12*1.01) $P_i$\\
(.18*1.01) $P_{i+1}$\\
(.12*.72) $Q_i$\\
(.18*.72) $Q_{i+1}$\\
(.07*1.01) $P_1$\\
(.4*1.01) $P_{2m}$\\
(.08*.72) $Q_1$\\
(.41*.72) $Q_{2m}$\\
\endSetLabels
\strut\AffixLabels{\includegraphics*[height=15cm]{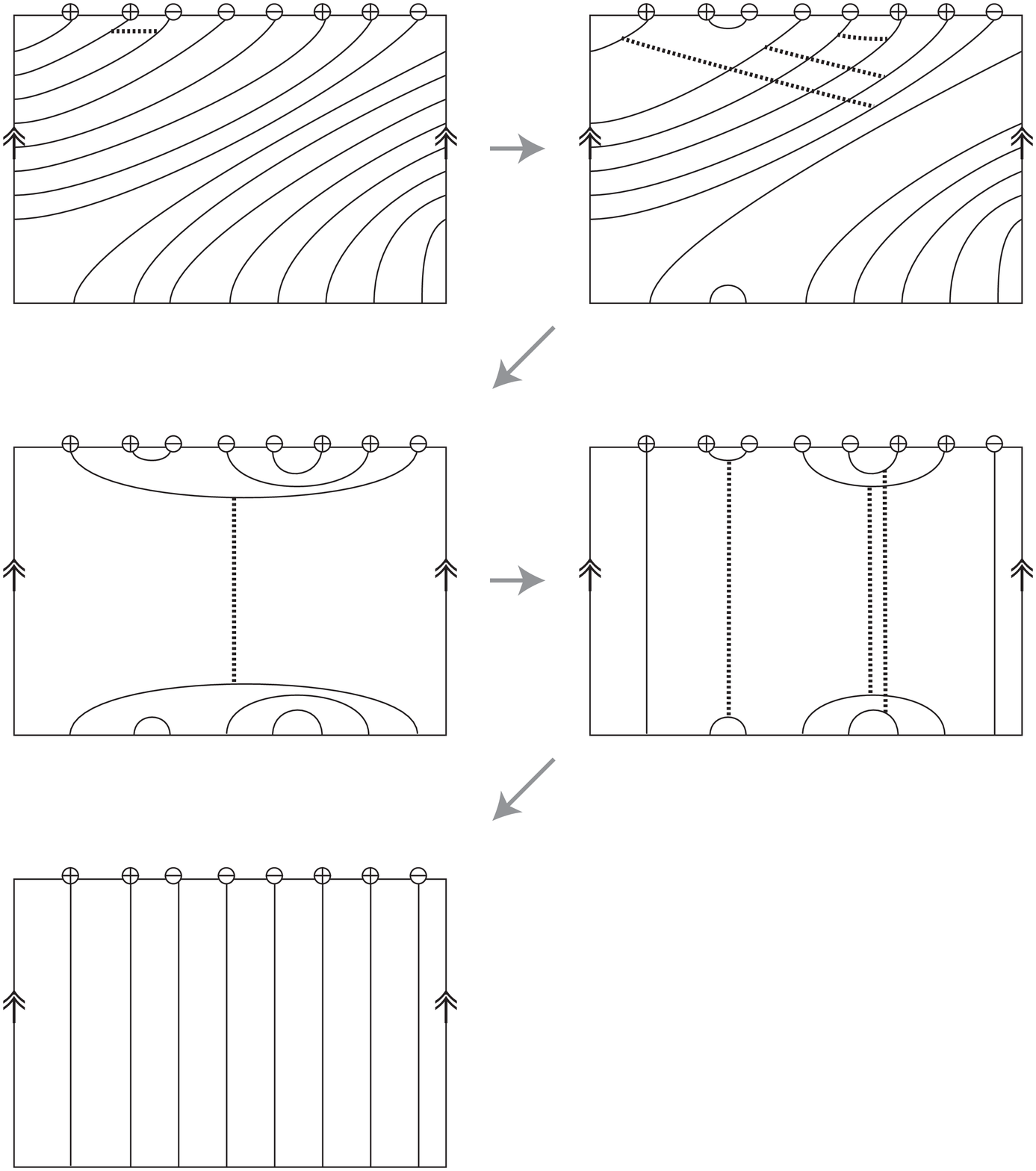}}
\caption{The b-arcs in $C\times I$ where $k=1$ and $m=4$. The a-arcs are omitted for simplicity.}\label{CxI-2}
\end{center}
\end{figure}

For $\e/3 \leq t < 2\e/3$ pair up nearby b-arcs of opposite orientations and join them by describing arcs (the thick dashed lines in Figures~\ref{CxI} and \ref{CxI-2}) then  change the configurations. 
More precisely, suppose that the consecutive elliptic points $P_i$ and $P_{i+1}$ have opposite signs. 
Join the b-arcs ending at $P_i$ and $P_{i+1}$ by a describing arc that does not intersect other b-arcs and a-arcs. 
After configuration change $P_i$ and $P_{i+1}$ are joined by a b-arc. 
Also $Q_i$ and $Q_{i+1}$ are joined by an arc, which possibly becomes sub-arc of a c-circle, but this does not cause any problem for the rest of the argument.

We repeat the above operation for the remaining b-arcs ending at $P_1,\dots, P_{i-1}, P_{i+2}, \dots, P_{2m}$. 
Because ${\bf b}_+ = {\bf b}_-$ all the elliptic points $P_1,\dots, P_{2m}$ can be paired up. 
Near $C\times \{0\}$ (resp. $C\times\{1\}$) we get $m$ nested, boundary parallel,   mutually disjoint b-arcs (resp. arcs).  
This construction introduces $m$ hyperbolic points. 

For $2\e/3 \leq t < \e$ join the same pairs by describing arcs and change the configurations. 
This gives another set of $m$ hyperbolic points. 
As a result we obtain $\Sigma' \cap S_\e = \Sigma \cap S_\e$.

We note that for each pair of b-arcs two hyperbolic points of opposite signs are introduced. 
Therefore the number $h_\pm$ (resp. $h_\pm'$) of $\pm$ hyperbolic points of $\Sigma$ (resp. $\Sigma'$), satisfy 
$$h_+' - h_+ = h_-' - h_-.$$
On the other hand, we do not introduce any new elliptic points during the above construction of $\Sigma'$. 
Hence the numbers of $\pm$ elliptic points for $\Sigma$ and $\Sigma'$ are the same. 
Using the formula $sl(\widehat b, [\Sigma]) = -(e_+-e_-) + (h_+ - h_-)$ the self-linking number does not change under this construction. 
\end{proof}

\section*{Acknowledgement}
The authors would like to thank the referees and Matt Hedden for comments on math and correcting typos. 
T.I. was supported by JSPS Research Grant-in-Aid for Research Activity Start-up (Grant Number 25887030).  
K.K. was partially supported by NSF grant DMS-1206770.

\end{document}